\documentclass[12pt,reqno]{amsart}
\usepackage{amssymb}
\usepackage{enumitem}
\usepackage{graphicx}

\makeatletter
\@namedef{subjclassname@2010}{%
  \textup{2010} Mathematics Subject Classification}
\makeatother

\usepackage[T1]{fontenc}
\usepackage{mathrsfs}
\usepackage{xcolor}
\usepackage[colorlinks, linkcolor=blue, citecolor=blue, urlcolor=blue]{hyperref}

\newtheorem{theorem}{Theorem}[section]
\newtheorem{corollary}[theorem]{Corollary}

\newtheorem{lemma}[theorem]{Lemma}

\theoremstyle{definition}

\DeclareMathOperator{\calA}{\mathcal{A}}
\DeclareMathOperator{\scrP}{\mathscr{P}}
\DeclareMathOperator{\symg}{sym_{\mathcal{G}}}

\DeclareMathOperator{\dom}{dom}
\DeclareMathOperator{\ran}{ran}

\begin{document}

\baselineskip=17pt

\title{A surjection from square onto power}

\author{Yinhe Peng}
\address{Academy of Mathematics and Systems Science\\
Chinese Academy of Sciences\\
East Zhong Guan Cun Road No.~55\\
Beijing 100190\\
People's Republic of China}
\email{pengyinhe@amss.ac.cn}

\author{Guozhen Shen}
\address{School of Philosophy\\
Wuhan University\\
No.~299 Bayi Road\\
Wuhan 430072\\
Hubei Province\\
People's Republic of China}
\email{shen\_guozhen@outlook.com}

\author{Liuzhen Wu}
\address{Academy of Mathematics and Systems Science\\
Chinese Academy of Sciences\\
East Zhong Guan Cun Road No.~55\\
Beijing 100190\\
People's Republic of China}
\email{lzwu@math.ac.cn}

\date{}

\begin{abstract}
It is shown that the existence of an infinite set $A$ such that
$A^2$ maps onto $2^A$ is consistent with $\mathsf{ZF}$.
\end{abstract}

\subjclass[2010]{Primary 03E35; Secondary 03E10, 03E25}

\keywords{permutation model, surjection, Cantor's theorem, axiom of choice}

\thanks{Peng was partially supported by NSFC No.~11901562 and a program of the Chinese Academy of Sciences.
Shen was partially supported by NSFC No.~12101466. Wu was partially supported by NSFC No.~11871464.}

\maketitle

\section{Introduction}

For a set $A$, let $2^A$ or $\scrP(A)$ denote the power set of $A$ (i.e., the set of all subsets of $A$),
let $A^2$ denote the Cartesian product $A\times A$, and let $[A]^2$ denote the set of all $2$-element subsets of $A$.

In~\cite{Cantor1892}, Cantor proves that, for all sets $A$, there are no surjections from $A$ onto $\scrP(A)$,
from which it follows that there are no injections from $\scrP(A)$ into $A$. In~\cite{Specker1954},
Specker proves in $\mathsf{ZF}$ (i.e., the Zermelo--Fraenkel set theory without the axiom of choice)
a generalization of the latter result, which states that, for all infinite sets $A$,
there are no injections from $\scrP(A)$ into $A^2$. Specker's result is trivial in the presence of the axiom of choice,
since it follows from the axiom of choice that there is a bijection between $A^2$ and $A$ for all infinite sets $A$.
It is natural to ask whether it is provable in $\mathsf{ZF}$ that,
\begin{equation}\label{sh01}
\text{for all infinite sets $A$, there are no surjections from $A^2$ onto $\scrP(A)$.}
\end{equation}
This question is known as the dual Specker problem and is asked in \cite{Truss1973}
(cf.~also \cite[p.~133]{Halbeisen2017} or \cite[Question~5.4]{ShenYuan2020}).
Recently, in~\cite{PengShen2022}, Peng and Shen prove in $\mathsf{ZF}$ a weaker version of \eqref{sh01},
which states that, for all infinite sets $A$, there are no surjections from $\omega\times A$ onto $\scrP(A)$.

In this paper, we prove by the method of permutation models that the existence of an infinite set $A$
such that there is a surjection from $[A]^2$ onto $\scrP(A)$ is consistent with $\mathsf{ZF}$.
This answers an open question proposed in~\cite{Halbeisen2018} (cf.~also \cite[Question~5.5]{ShenYuan2020}),
and gives a negative answer to the dual Specker problem, since clearly there is a surjection from $A^2$ onto $[A]^2$.

\section{Permutation models}

We refer the readers to \cite[Chap.~8]{Halbeisen2017} or \cite[Chap.~4]{Jech1973}
for an introduction to the theory of permutation models.
Permutation models are not models of $\mathsf{ZF}$;
they are models of $\mathsf{ZFA}$ (i.e., the Zermelo--Fraenkel set theory with atoms).
Nevertheless, they indirectly give, via the Jech--Sochor theorem
(cf.~\cite[Theorem~17.2]{Halbeisen2017} or \cite[Theorem~6.1]{Jech1973}), models of $\mathsf{ZF}$.

Let $A$ be the set of atoms, let $\mathcal{G}$ be a group of permutations of $A$,
and let $\mathfrak{F}$ be a normal filter on $\mathcal{G}$.
We write $\symg(x)$ for the set $\{\pi\in\mathcal{G}\mid\pi x=x\}$,
where $\pi\in\mathcal{G}$ extends to a permutation of the universe by
\[
\pi x=\{\pi y\mid y\in x\}.
\]
Then $x$ belongs to the permutation model $\mathcal{V}$ determined by $\mathcal{G}$ and $\mathfrak{F}$
if and only if $x\subseteq\mathcal{V}$ and $\symg(x)\in\mathfrak{F}$.

\section{A new permutation model}
In this section, we construct a permutation model in which
there exists an infinite set $A$ such that $[A]^2$ maps onto $\scrP(A)$.
For convenience, we shall work in $\mathsf{ZFA}+\mathsf{AC}+\forall n\in\omega(2^{\aleph_n}=\aleph_{n+1})$.
The strategy of our construction is as follows:

We construct a sequence of functions $h_n$ ($n\in\omega$) so that each $h_n$
is a surjection from $[\omega_{n+1}\setminus\omega_n]^2$ onto $\scrP(\omega_n)\times\scrP(\omega)$
and satisfies a certain kind of back-and-forth condition (see the next lemma).
Then we take $A=\omega_\omega$, $h=\bigcup_{n\in\omega}h_n$,
let $\mathcal{G}$ be the group of all ``automorphisms'' of $\langle A,h\rangle$,
and let $\mathfrak{F}$ be the normal filter on $\mathcal{G}$ generated
by those subgroups of $\mathcal{G}$ that leave some $\omega_n$ fixed
(which will be made clear in the sequel).
The back-and-forth condition guarantees that $\langle A,h\rangle$ is in a sense ultrahomogeneous;
that is, every partial automorphism of $\langle A,h\rangle$ of smaller cardinality extends to an automorphism.
By the ultrahomogeneity of $\langle A,h\rangle$, we can conclude that,
in the permutation model determined by $\mathcal{G}$ and $\mathfrak{F}$,
every subset of $A$ can be encoded by an element of $\scrP(\omega_n)\times\scrP(\omega)$ for some $n\in\omega$,
and hence $h$ induces a surjection from $[A]^2$ onto $\scrP(A)$.

\begin{lemma}\label{sh02}
For every $n\in\omega$, there exists a surjection
$h_n:[\omega_{n+1}\setminus\omega_n]^2\twoheadrightarrow\scrP(\omega_n)\times\scrP(\omega)$
satisfying the following condition: For all $\alpha<\omega_{n+1}$,
all injections $f:\alpha\to\omega_{n+1}\setminus\omega_n$,
and all functions $g:\alpha\to\scrP(\omega_n)\times\scrP(\omega)$,
there exists a $\beta\in\omega_{n+1}\setminus(\omega_n\cup\ran(f))$
such that $h_n(\{\beta,f(\gamma)\})=g(\gamma)$ for every $\gamma<\alpha$.
\end{lemma}
\begin{proof}
Let $n\in\omega$. We define $\langle X_\delta\rangle_{\delta<\omega_{n+1}}$ and $\langle p_\delta\rangle_{\delta<\omega_{n+1}}$
by recursion as follows.
\begin{itemize}[leftmargin=*]
\item $X_0=p_0=\varnothing$.
\item Let $\delta<\omega_{n+1}$ and assume that $X_\delta$ and $p_\delta$ have been defined so that
$|X_\delta|\leqslant\aleph_{n+1}$ and $p_\delta:[X_\delta]^2\to\scrP(\omega_n)\times\scrP(\omega)$.
We use $\varphi(\delta,\alpha,f,g)$ to denote the following formula:
\[
\alpha<\omega_{n+1},f:\alpha\to X_\delta\text{ is injective, and }g:\alpha\to\scrP(\omega_n)\times\scrP(\omega).
\]
For all $\alpha,f,g$ such that $\varphi(\delta,\alpha,f,g)$, choose a $z_{\delta,\alpha,f,g}\notin X_\delta$ so that
\[
z_{\delta,\alpha,f,g}=z_{\delta,\alpha',f',g'}\text{ only if }\alpha=\alpha',f=f',\text{ and }g=g'.
\]
Define
\[
X_{\delta+1}=X_\delta\cup\{z_{\delta,\alpha,f,g}\mid\varphi(\alpha,f,g)\},
\]
and let $p_{\delta+1}$ be the function from $[X_{\delta+1}]^2$ to $\scrP(\omega_n)\times\scrP(\omega)$ defined by
\[
p_{\delta+1}(\{x,y\})=
\begin{cases}
p_\delta(\{x,y\})                     & \text{if $\{x,y\}\in[X_\delta]^2$,}\\
g(\gamma)                             & \text{if $\{x,y\}=\{z_{\delta,\alpha,f,g},f(\gamma)\}$ for some $\gamma<\alpha$,}\\
\langle\varnothing,\varnothing\rangle & \text{otherwise.}
\end{cases}
\]
Since we are working in $\mathsf{ZFA}+\mathsf{AC}+\forall n\in\omega(2^{\aleph_n}=\aleph_{n+1})$,
it follows that $|X_{\delta+1}|\leqslant\aleph_{n+1}$.
\item If $\delta<\omega_{n+1}$ is a limit ordinal, define
\[
X_\delta=\bigcup_{\gamma<\delta}X_\gamma\qquad\text{and}\qquad p_\delta=\bigcup_{\gamma<\delta}p_\gamma.
\]
\end{itemize}
Finally, we define
\[
X=\bigcup_{\delta<\omega_{n+1}}X_\delta\qquad\text{and}\qquad p=\bigcup_{\delta<\omega_{n+1}}p_\delta.
\]
Clearly, $|X|=\aleph_{n+1}$ and $p$ is a surjection from $[X]^2$ onto $\scrP(\omega_n)\times\scrP(\omega)$.
Furthermore, for all $\alpha<\omega_{n+1}$, all injections $f:\alpha\to X$,
and all functions $g:\alpha\to\scrP(\omega_n)\times\scrP(\omega)$,
since $\omega_{n+1}$ is regular, there is a $\delta<\omega_{n+1}$ such that $\ran(f)\subseteq X_\delta$,
and thus $\varphi(\delta,\alpha,f,g)$, which implies that $z_{\delta,\alpha,f,g}$ is an element of $X$
such that $p(\{z_{\delta,\alpha,f,g},f(\gamma)\})=p_{\delta+1}(\{z_{\delta,\alpha,f,g},f(\gamma)\})=g(\gamma)$ for every $\gamma<\alpha$.
Finally, since $|X|=|\omega_{n+1}\setminus\omega_n|$,
we can, without loss of generality, identify $X$ with $\omega_{n+1}\setminus\omega_n$ and take $h_n=p$.
\end{proof}

For each $n\in\omega$, let $\calA_n=\langle\omega_{n+1},h_n\rangle$.
A \emph{partial automorphism} of $\calA_n$ is an injection $\pi$ from a subset of $\omega_{n+1}$ into $\omega_{n+1}$
satisfying the following two conditions:
\begin{itemize}[leftmargin=*]
\item $\omega_n\subseteq\dom(\pi)$ and $\pi[\omega_n]=\omega_n$;
\item for all $\{\xi,\eta\}\in[\dom(\pi)\setminus\omega_n]^2$,
      if $h_n(\{\xi,\eta\})=\langle B,r\rangle$, where $B\subseteq\omega_n$ and $r\subseteq\omega$,
      then $h_n(\{\pi(\xi),\pi(\eta)\})=\langle\pi[B],r\rangle$.
\end{itemize}
An \emph{automorphism} of $\calA_n$ is a partial automorphism which is a permutation of $\omega_{n+1}$.

\begin{lemma}\label{sh03}
For every $n\in\omega$, each partial automorphism $\pi_0$ of $\calA_n$
with $|\dom(\pi_0)|<\aleph_{n+1}$ extends to an automorphism of $\calA_n$.
\end{lemma}
\begin{proof}
Let $n\in\omega$ and let $\pi_0$ be a partial automorphism of $\calA_n$ such that $|\dom(\pi_0)|<\aleph_{n+1}$.
We define an increasing sequence $\langle\pi_\delta\rangle_{\delta<\omega_{n+1}}$
of partial automorphisms of $\calA_n$ by recursion as follows.
\begin{itemize}[leftmargin=*]
\item Let $\delta<\omega_{n+1}$ and assume that
$\pi_\delta$ is a partial automorphism of $\calA_n$ such that $|\dom(\pi_\delta)|<\aleph_{n+1}$.

\noindent
\textsc{Case}~1 (back). $\delta$ is even. Let $\eta$ be the least element of $\omega_{n+1}\setminus\ran(\pi_\delta)$.
Let $\alpha$ and $f$ be the order-type and the enumerating function of $\dom(\pi_\delta)\setminus\omega_n$, respectively.
For each $\gamma<\alpha$, suppose that $h_n(\{\eta,\pi_\delta(f(\gamma))\})=\langle B_\gamma,r_\gamma\rangle$,
where $B_\gamma\subseteq\omega_n$ and $r_\gamma\subseteq\omega$.
Let $g$ be the function from $\alpha$ to $\scrP(\omega_n)\times\scrP(\omega)$
defined by
\[
g(\gamma)=\langle\pi_\delta^{-1}[B_\gamma],r_\gamma\rangle.
\]
Then, by Lemma~\ref{sh02}, there is a least $\beta\in\omega_{n+1}\setminus\dom(\pi_\delta)$
such that $h_n(\{\beta,f(\gamma)\})=g(\gamma)$ for every $\gamma<\alpha$. Define
\[
\pi_{\delta+1}=\pi_\delta\cup\{\langle\beta,\eta\rangle\}.
\]
Then $\pi_{\delta+1}$ is a partial automorphism of $\calA_n$,
since $\pi_\delta$ is a partial automorphism and for all $\zeta\in\dom(\pi_\delta)\setminus\omega_n$ we have
\[
h_n(\{\beta,\zeta\})=h_n(\{\beta,f(\gamma)\})=g(\gamma)=\langle\pi_{\delta+1}^{-1}[B_\gamma],r_\gamma\rangle
\]
and
\[
h_n(\{\pi_{\delta+1}(\beta),\pi_{\delta+1}(\zeta)\})=h_n(\{\eta,\pi_\delta(f(\gamma))\})=\langle B_\gamma,r_\gamma\rangle
\]
where $\zeta=f(\gamma)$ for some $\gamma<\alpha$.

\noindent
\textsc{Case}~2 (forth). $\delta$ is odd. Let $\xi$ be the least element of $\omega_{n+1}\setminus\dom(\pi_\delta)$.
Let $\alpha$ and $f$ be the order-type and the enumerating function of $\ran(\pi_\delta)\setminus\omega_n$, respectively.
For each $\gamma<\alpha$, suppose that $h_n(\{\xi,\pi_\delta^{-1}(f(\gamma))\})=\langle B_\gamma,r_\gamma\rangle$,
where $B_\gamma\subseteq\omega_n$ and $r_\gamma\subseteq\omega$.
Let $g$ be the function from $\alpha$ to $\scrP(\omega_n)\times\scrP(\omega)$
defined by
\[
g(\gamma)=\langle\pi_\delta[B_\gamma],r_\gamma\rangle.
\]
Then, by Lemma~\ref{sh02}, there is a least $\beta\in\omega_{n+1}\setminus\ran(\pi_\delta)$
such that $h_n(\{\beta,f(\gamma)\})=g(\gamma)$ for every $\gamma<\alpha$. Define
\[
\pi_{\delta+1}=\pi_\delta\cup\{\langle\xi,\beta\rangle\}.
\]
Then $\pi_{\delta+1}$ is a partial automorphism of $\calA_n$,
since $\pi_\delta$ is a partial automorphism and for all $\zeta\in\dom(\pi_\delta)\setminus\omega_n$ we have
\[
h_n(\{\xi,\zeta\})=h_n(\{\xi,\pi_\delta^{-1}(f(\gamma))\})=\langle B_\gamma,r_\gamma\rangle
\]
and
\[
h_n(\{\pi_{\delta+1}(\xi),\pi_{\delta+1}(\zeta)\})=h_n(\{\beta,f(\gamma)\})=g(\gamma)=\langle\pi_{\delta+1}[B_\gamma],r_\gamma\rangle
\]
where $\pi_\delta(\zeta)=f(\gamma)$ for some $\gamma<\alpha$.
\item If $\delta<\omega_{n+1}$ is a limit ordinal, define
\[
\pi_\delta=\bigcup_{\gamma<\delta}\pi_\gamma.
\]
\end{itemize}
Finally, we define
\[
\pi=\bigcup_{\delta<\omega_{n+1}}\pi_\delta.
\]
Clearly, $\pi$ is a partial automorphism of $\calA_n$ with $\dom(\pi)=\ran(\pi)=\omega_{n+1}$,
and hence is an automorphism of $\calA_n$.
\end{proof}

Let $\calA=\langle\omega_\omega,h\rangle$, where $h=\bigcup_{n\in\omega}h_n$.
An \emph{automorphism} of $\calA$ is a permutation $\pi$ of $\omega_\omega$
satisfying the following two conditions:
\begin{itemize}[leftmargin=*]
\item for all $n\in\omega$, $\pi[\omega_n]=\omega_n$;
\item for all $n\in\omega$ and all $\{\xi,\eta\}\in[\omega_{n+1}\setminus\omega_n]^2$,
      if $h(\{\xi,\eta\})=\langle B,r\rangle$, where $B\subseteq\omega_n$ and $r\subseteq\omega$,
      then $h(\{\pi(\xi),\pi(\eta)\})=\langle\pi[B],r\rangle$.
\end{itemize}

\begin{lemma}\label{sh04}
For all $n\in\omega$ and all $\xi,\eta\in\omega_{n+1}\setminus\omega_n$ with $\xi\neq\eta$,
there exists an automorphism of $\calA$ which fixes $\omega_n$ pointwise and moves $\xi$ to $\eta$.
\end{lemma}
\begin{proof}
Let $n\in\omega$ and let $\xi,\eta\in\omega_{n+1}\setminus\omega_n$ be such that $\xi\neq\eta$.
We define $\langle\pi_k\rangle_{k\in\omega}$ by recursion as follows.
\begin{itemize}[leftmargin=*]
\item By Lemma~\ref{sh03}, $\mathrm{id}_{\omega_n}\cup\{\langle\xi,\eta\rangle\}$
extends to an automorphism $\pi_0$ of $\calA_n$, where
$\mathrm{id}_{\omega_n}$ is the identity permutation of $\omega_n$.
\item Let $k\in\omega$ and assume that $\pi_k$ is an automorphism of $\calA_{n+k}$.
By Lemma~\ref{sh03}, $\pi_k$ extends to an automorphism $\pi_{k+1}$ of $\calA_{n+k+1}$.
\end{itemize}
Finally, we define
\[
\pi=\bigcup_{k\in\omega}\pi_k.
\]
Clearly, $\pi$ is an automorphism of $\calA$ which fixes $\omega_n$ pointwise and moves $\xi$ to $\eta$.
\end{proof}

Now, we construct our permutation model.
For each $\alpha<\omega_\omega$, we assign a new atom $a_\alpha$
and define the set $A$ of atoms by $A=\{a_\alpha\mid\alpha<\omega_\omega\}$.
However, for the sake of simplicity, we shall identify $A$ with $\omega_\omega$.
Now, let $\mathcal{G}$ be the group of all automorphisms of $\calA$,
let $\mathcal{G}_n=\{\pi\in\mathcal{G}\mid\forall\alpha<\omega_n(\pi(\alpha)=\alpha)\}$ for each $n\in\omega$,
and let $\mathfrak{F}$ be the normal filter on $\mathcal{G}$
generated by the subgroups $\mathcal{G}_n$ ($n\in\omega$) of $\mathcal{G}$.
Let $\mathcal{V}$ be the permutation model determined by $\mathcal{G}$ and $\mathfrak{F}$.
Clearly, $h\in\mathcal{V}$.

\begin{lemma}\label{sh05}
In $\mathcal{V}$, for every subset $C$ of $A$, there exists an $m\in\omega$
such that, for all $n\geqslant m$, either $\omega_{n+1}\setminus\omega_n\subseteq C$
or $C\cap(\omega_{n+1}\setminus\omega_n)=\varnothing$.
\end{lemma}
\begin{proof}
Let $C\in\mathcal{V}$ be a subset of $A$.
Since $\symg(C)\in\mathfrak{F}$, there exists an $m\in\omega$ such that $\mathcal{G}_m\subseteq\symg(C)$;
that is, every $\pi\in\mathcal{G}$ fixing $\omega_m$ pointwise also fixes $C$.
Assume toward a contradiction that there is an $n\geqslant m$ such that
$\omega_{n+1}\setminus\omega_n\nsubseteq C$ and $C\cap(\omega_{n+1}\setminus\omega_n)\neq\varnothing$.
Let $\xi,\eta\in\omega_{n+1}\setminus\omega_n$ be such that $\xi\in C$ and $\eta\notin C$.
By Lemma~\ref{sh04}, there is a $\pi\in\mathcal{G}$ which fixes $\omega_n$ pointwise and moves $\xi$ to $\eta$.
Since $\omega_m\subseteq\omega_n$, it follows that $\pi$ fixes $C$,
contradicting the fact that $\pi$ moves an element $\xi$ in $C$ to an element $\eta$ outside $C$.
\end{proof}

\begin{lemma}\label{sh06}
In $\mathcal{V}$, there exists a surjection from $[A]^2$ onto $\scrP(A)$.
\end{lemma}
\begin{proof}
Let $\Phi$ be the function on $[A]^2$ defined by
\[
\Phi(\{\xi,\eta\})=B\cup\bigcup_{k\in r}(\omega_{n+k+1}\setminus\omega_{n+k}),
\]
if $\{\xi,\eta\}\in[\omega_{n+1}\setminus\omega_n]^2$ for some $n\in\omega$
and $h(\{\xi,\eta\})=\langle B,r\rangle$;
otherwise, $\Phi(\{\xi,\eta\})=\varnothing$.
Since $h\in\mathcal{V}$, it follows that $\Phi\in\mathcal{V}$.
Now, it follows from Lemma~\ref{sh05} that, in $\mathcal{V}$,
$\Phi$ is a surjection from $[A]^2$ onto $\scrP(A)$.
\end{proof}

Now the next theorem immediately follows from Lemma~\ref{sh06} and the Jech--Sochor theorem.
\begin{theorem}\label{sh07}
It is consistent with $\mathsf{ZF}$ that there exists an infinite set $A$ such that $[A]^2$ maps onto $\scrP(A)$.
\end{theorem}

As a corollary of Theorem~\ref{sh07}, we obtain a negative answer to the dual Specker problem.

\begin{corollary}
It is consistent with $\mathsf{ZF}$ that there exists an infinite set $A$ such that $A^2$ maps onto $2^A$.
\end{corollary}

%\subsection*{Acknowledgements}
%We would like to give thanks to an anonymous referee
%for catching some errors and making useful suggestions.

%%%%%%%%%%%%%%%%%%%%%%%%%%%%%%%%%%%

\end{document}